\documentclass[10pt,leqno]{amsart}
\usepackage{amsmath, latexsym,  amssymb, amscd}
\usepackage[all]{xy}
\usepackage[british]{babel}
\usepackage{url}

\usepackage{enumitem}

\usepackage{hyperref}

\DeclareMathOperator{\codim}{codim}

\newcommand{\abs}[1]{\left| #1 \right|}

\newcommand{\torsione}{\zeta}

\newcommand{\Ci}{\mathcal{C}}

\newcommand{\gum}{\mathbb{G}}

\newcommand{\sotto}{B}

\newcommand{\cV}{(h(V)+\deg V)[k_\mathrm{tor}(V):k_\mathrm{tor}]}
\newcommand{\cC}{(h( \Ci)+\deg  \Ci)[k_\mathrm{tor}( \Ci):k_\mathrm{tor}]}

\newcommand{\cod}{\mathrm{codim} \,}

\newcommand{\qe}{\mathbb{Q}}

\renewcommand{\epsilon}{\varepsilon}

\newtheorem{thm}{Theorem}[section]

\newtheorem{con}[thm]{Conjecture}

\newtheorem{D}[thm]{Definition}

\newtheorem{remark}[thm]{Remark}

\newcommand{\ambiente}{G}

\newtheorem*{TAC}{Torsion Anomalous Conjecture (TAC)}
\newtheorem*{CIT}{Conjecture on Intersection with Torsion Varieties (CIT)}

\renewcommand{\P}{\mathbb{P}}
\newcommand{\Z}{\mathbb{Z}}

\newcommand{\Q}{\mathbb{Q}}

\title{Explicit height bounds and the effective Mordell-Lang Conjecture}
\author[{ Evelina Viada}]{  
 }

\begin{document}


\begin{abstract}
We first give an overview of  some  landmark theorems and recent conjectures in Diophantine Geometry.  In the elliptic case, we  prove some new bounds for  torsion anomalous points and we clarify the implications of several height bounds on the effective Mordell-Lang Conjecture.   In addition, we explicitly  bound the N\'eron-Tate height of the rational points of a new  family of curves of increasing genus, proving the effective Mordell Conjecture for these curves.
\end{abstract}
\maketitle

{\centerline{Evelina Viada    \footnote{Viada's work suported by the FNS Fonds National Suisse.}}}

\section*{Structure of the paper}
In the first chapter we recall some landmark theorems of the last century, such as the Mordell-Weil Theorem, the Manin-Mumford Conjecture, the Mordell-Lang Conjecture and the Bogomolov  Conjecture. In section 2, we introduce the definitions of anomalous and torsion anomalous varieties and a general open conjecture: the Torsion Anomalous Conjecture. In section 3, we discuss  effectivity aspects, in particular we present conjectures and results on the height of  torsion anomalous varieties. In section 4, we explain how these bounds imply some cases of the effective Mordell-Lang Conjecture, which is one of the challenge of this century. As a further  application, in section 5  we {\bf explicitly}  bound the N\'eron-Tate height of the rational points of a new  family of curves of increasing genus, proving the effective Mordell Conjecture for these curves. In section 6, we prove new bounds for the torsion anomalous points  on a curve and we explain the implications of these bounds on the quantitative and effective Mordell-Lang Conjecture.

\section{Setting and Classical Results}
 By  a variety $X$ over a number field $k$ we mean a subset of an affine or projective space, defined by a set of polynomial equations with coefficients in $k$.  The $k$-rational points of $X$ are the solutions with coordinates in $k$, denoted $X(k)$. We denote by $G$  a torus or an abelian variety over the algebraic numbers and we identify $G$ with $G(\overline{\qe})$.   A torus $\gum^n$  is the affine algebraic group $(\overline{\qe}^*)^n$ endowed with coordinatewise multiplication. In $\gum^n$ we consider  the Weil logarithmic height. By an abelian variety $A$, we mean an irreducible  group variety over the algebraic numbers embedded in some projective space. On $A$  we consider the canonical  N\'eron-Tate logarithmic height. We denote by $\hat{h}$ the described height function. We consider on $G$ the Zariski topology. The torsion points of $G$, denoted by ${\rm Tor}_G$,   form a dense subset of  $G$ and are defined over $\overline{\qe}$. By a  generalisation of  Kronecker's theorem, ${\rm Tor}_G$ is exactly the set of points of height zero.  We denote by $k_{\rm tor}$ the field of definition of  ${\rm Tor}_G$.\\

Let $V$ be a subvariety of $G$ defined over the algebraic numbers.

 A central question in Diophantine Geometry is to describe the set of points on $V$ that satisfy some natural arithmetic property. The general method to study such a problem is to compare the arithmetic and the geometry of the variety.
 The  leading principle is that to a geometric assumption on $V$ corresponds the non-density of subsets of $V$ defined via arithmetic properties.

In recent years several authors have obtained crucial results in this context, which in turn triggered a host of quite general questions concerning subvarieties embedded in general group varieties. \\

 A  simple but illustrative example is the case of an irreducible curve $ \Ci$ in a torus $\gum^n$. Inside the torus we can consider natural subsets defined via the group law: the torsion subgroup   ${\rm Tor}_{\gum^n}$, any subgroup $\Gamma$ of finite rank and the union $\mathcal{B}_r$ of all algebraic subgroups of codimension $\ge r$. A natural question is to ask when the intersection of the curve with one of these sets is finite. Clearly, if the curve is a component of an algebraic subgroup none of these intersections is finite, however, excluding the trivial cases, it turns out that the intersection of $ \Ci$ with  ${\rm Tor}_{\gum^n}$, $\Gamma$ and $\mathcal{B}_2$ is finite. These are very special cases of the Manin-Mumford Conjecture, Mordell-Lang Conjecture and of the Torsion Anomalous Conjecture, discussed below.\\

 Geometrically a variety is characterized by its `distance' from being an algebraic subgroup.  More precisely:
\begin{D} A variety $V\subset G$ is  {\em a torsion variety (}resp. {\em a translate)} if it is a finite union of translates of algebraic subgroups of $G$ by torsion points (resp. by points).

 A proper irreducible  subvariety $V\subsetneq G$  is  {\em  weak-transverse (}resp. {\em transverse)} if it is not contained in any proper torsion variety (resp. any proper translate).\end{D}
Clearly transverse implies weak-transverse and  non translate, which imply non-torsion.

Arithmetically one considers subsets of $V$ defined via the group structure or via the Weil or N\'eron-Tate height function, for instance torsion points, subgroups of finite rank or points of small height. For a subvariety $V\subset G\subset \mathbb{P}^m$  we {consider} the  normalised height of $V$, denoted $h(V)$, {defined}  in terms of the Chow form of the ideal of $V$, as done  by P. Philippon in \cite{patriceI} and \cite{patrice}. \\

We now recall some  of the classical landmark theorems proven in the last century in this context.
\begin{thm}{\rm [Manin-Mumford Conjecture]}
\label{MM}~ A subvariety $V$ of $G$ is a torsion variety if and only if the set $V\cap {\rm Tor}_\ambiente$ is Zariski dense in $V$. 
\end{thm}

The Manin-Mumford Conjecture
was proved by M. Raynaud \cite{ra2} for abelian varieties and  by  M. Laurent 
{\cite{lau} 
for tori. 
  Another break-through result is
\begin{thm}\label{MC}
 \label{M}{\rm  [Mordell Conjecture]}
 There are only finitely many    $k$-rational points on a curve over  a number field $k$ and of genus at least $2$.
\end{thm}
G. Faltings  \cite{Fa1} used sophisticated  tools involving N\'eron models and    Galois representations, to prove the Mordell Conjecture.
 P. Vojta \cite{vojta} provided a completely different proof based on the so-called Vojta inequality. Using deep Arakelov intersection theory, he showed that if two points of  `big' height lie on a curve, then they  have a fairly big angle with respect to a height norm.
Vojta's work is fundamental in the proof of G. Faltings  \cite{Fa2} of Lang's conjecture, which implies the general Mordell Conjecture via the Mordell-Weil Theorem: the set of $k$-rational points of an abelian variety over $k$ is a finitely generated group.
 \begin{thm}{\em [Mordell-Lang Conjecture (MLC)]}
 \label{ML}
Let $\Gamma$ be a subgroup of $\ambiente$ of finite rank.  Suppose that $V\subset G$ is not a translate. Then,
 the set $V\cap \Gamma$  is  not Zariski dense in $V$.

\end{thm}
 As mentioned, G. Faltings  proved it for $\Gamma$  a  finitely generated subgroup in an abelian variety (Lang's Conjecture).      Further  work of  P. Vojta  \cite{vojta}, 
 M. McQuillan 
 \cite{mcq} 
 and M. Hindry
  \cite{marc} 
 finally proved the general case of varieties in semi-abelian varieties and $\Gamma$ of finite rank.
Note that the Mordell-Lang Conjecture implies  the Manin-Mumford Conjecture, because $\mathrm{Tor}_G$ is a subgroup of rank zero. \smallskip

  In  1981,  F. Bogomolov \cite{bogomolov} stated a conjecture for curves, which generalises as follows. Let $V(\theta)$ be the set of points of $V$ of height at most $\theta$ and $ \overline{V(\theta)}$ its Zariski closure.
\begin{thm}{\em [Bogomolov Conjecture]}
A variety $V\subset G$ is non torsion if and only if its essential minimum  $$\mu(V):=\inf\{\theta \in\mathbb{R}\,\,\,:\,\,\, \overline{V(\theta)}=V\}$$ is strictly positive.
\end{thm}
This generalises the  Manin-Mumford Conjecture, because the torsion points  are precisely the points of  height zero. The Bogomolov Conjecture was proved by   E. Ullmo \cite{ulmo} for curves and by S. Zhang \cite{ZhangEquidistribution} in general. \smallskip

Effective and explicit versions of this conjecture have important implications, as clarified below. 
For points, the  Bogomolov Conjecture becomes a generalisation of Kronecker's Theorem, and explicit versions are  a  generalisation (up to a $\log$ term) of a famous conjecture of 
E.  and D.-H. Lehmer (see \cite{lehmer}):  for an algebraic number $\alpha$ not a root of unity 
$
h(\alpha)\geq\frac{c}{[\qe( \alpha):\qe]},
$ with $c$ an explicit constant. The best result proved for algebraic numbers is Dobrowolski's  bound
 \cite{dob}
: $$
h(\alpha)\geq\frac{c}{[\qe( \alpha):\qe]}\left(\frac{\log [\qe( \alpha):\qe]}{\log\log[\qe( \alpha):\qe]}\right)^{-3}.$$
In general, we state the following conjecture:
\begin{con}{\em [Lehmer Type Bound] }
\label{lemer} Assume that  the  group variety $G$ is defined over a number field $k$.
Let $\alpha $ be a  point of $G$ and let $B$ be  the irreducible torsion variety of minimal dimension containing $\alpha$. Then,
 for any real $\eta>0$, there exists a positive  constant $c(\ambiente,\eta)>0$ such that  $$h(\alpha)\ge c(\ambiente,\eta)\frac{(\deg B)^{\frac{1}{\dim B}-\eta}}{[k_\mathrm{tor}(\alpha):k_\mathrm{tor}]^{\frac{1}{\dim B}+ \eta}},$$
 where $k_\mathrm{tor}$ is the field of definition of the torsion of $G$.
\end{con}

 F. Amoroso and  S. David  \cite{Amo-Dav 1999} 
   proved the conjecture in the toric case for weak-transverse points. F. Amoroso and myself proved a general explicit result in Tori \cite{commentari}. S. David and M. Hindry \cite{davidhindry} proved this bound for weak-transverse points in abelian varieties with  complex multiplication (in short CM). M. Carrizzosa  \cite{carrizosaIMRN} proved it for points in CM abelian varieties.
   No known method  generalises such a sharp bound in a general abelian variety. This is a big obstruction when applying the bound. Some weaker bounds in this case are given for instance by D. Masser \cite{Masser} and M. Baker and J. Silverman \cite{BakSil}.
\smallskip\\
In positive dimension,  an effective version of the Bogomolov Conjecture is 
\begin{con}{\em [Bogomolov Type Bound]}
\label{bog}
Assume that  $X\subset G$ is irreducible and not a translate. Let $H$ be the irreducible translate of minimal dimension containing $X$. Then,  for any real $\eta>0$, 
there exists an effective  positive constant $c(\ambiente,\eta)$  such that 
$$
\mu(X)\geq c(\ambiente,\eta)\frac{(\deg H)^{\frac{1}{\cod_H X}-\eta}}{(\deg X)^{\frac{1}{\cod_H X}+\eta}},
$$
where $\cod_H X$ is the codimension of $X$ in $H$.
\end{con}
   This   was proved for transverse varieties in a torus by  F. Amoroso and S. David \cite{Amo-Dav 2003}. The sharpest explicit bound in tori is given by  F. Amoroso and myself in \cite{iofrancesco}. We avoided a complicated descent method used before. Instead,  we used a geometric induction which is simple and optimal. A. Galateau   \cite{galateau} proved this bound for transverse varieties in abelian varieties with a positive density of ordinary primes. In \cite{funtoriale} with S. Checcoli and F. Veneziano, we  deduce from Galateau's bound the Bogomolov bound for non-translates.

\section{Anomalous and Torsion Anomalous Intersections}
\bigskip
 Anomalous Intersections are a natural framework for these kind of problems.
  \begin{D} An irreducible subvariety $Y$ of $V\subset G$ is $V$-torsion anomalous if:
 \begin{itemize}
\item[(i)] $Y$  is a component of  the intersection of $V$ with  a proper torsion variety $B$ of $G$; and
\item[(ii)] $Y$ has  dimension larger than expected, i.e. $${\rm{codim }}Y< {\rm{codim }} V+{\rm{codim }} B.$$
  \end{itemize} In addition, $Y$ of positive dimension  is $V$-anomalous if the above conditions are  satisfied with $B$ a proper translate of $G$.

We denote by $V^{ta}$ the set of points of $V$ deprived of the $V$-torsion anomalous varieties and by $V^{oa}$ the set of points of $V$ deprived of the $V$-anomalous varieties.

We say that $B$ is \emph{minimal} for $Y$ if it satisfies (i) and (ii) and  it has minimal dimension. The codimension of $Y$ in its minimal $B$ is called the \emph{relative codimension} of $Y$.

We also say that a $V$-(torsion) anomalous variety $Y$ is \emph{maximal}  if it is not contained in any $V$-(torsion) anomalous variety of strictly larger dimension. \end{D} 

We note that we define  torsion anomalous points, while in the original definitions they were excluded. Our choice makes  some statements easier. It is clear that any point would be anomalous, so in this case the dimension must be positive.
This point of view,  introduced by E. Bombieri, D. Masser and U. Zannier   \cite{BMZ_tori}, has been successfully exploited in the development of the theory, and still leads the research in this field. The following conjecture has been open for several years:
\begin{TAC}  For an irreducible variety $V\subset G$, there are only finitely many maximal $V$-torsion anomalous varieties.
\end{TAC}
 The special case of considering only   the $V$-torsion anomalous varieties that  come from an intersection expected to be empty, was considered previously by B. Zilber \cite{Zilber}. So the TAC implies:
 \begin{CIT}  If $V\subset G$ is weak-transverse, then the intersection of $V$ with the union of all algebraic subgroups of codimension at least $\dim V+1$  is non dense in $V$. 
\end{CIT}
  The CIT, as well as the TAC, implies several celebrated questions such as the Manin-Mumford Conjecture and  the Mordell-Lang Conjecture; they are  related to  the Bogomolov Conjecture and  to previous investigations by S. Lang, P. Liardet, F. Bogomolov, S. Zhang and others. Further connections to  model theory have arisen when similar issues were considered by B. Zilber.  Recent works in the context of the  Morton and Silverman Conjectures highlight an important  interplay between Anomalous Intersections and Algebraic Dynamics (for an introduction to the subject see \cite{Sil.2}).

The TAC  has been only partially answered, after having received great attention by many mathematicians. Only the following few cases of the TAC are known:  for curves  in a product of elliptic curves (E. Viada  \cite{ioant}), in abelian varieties with CM (G. R\'emond \cite{RemondIII}), in abelian varieties with a positive density of ordinary primes (E. Viada \cite{ijnt}), in any abelian variety (P. Habegger and J. Pila \cite{PP}) and in a torus  (G. Maurin  \cite{Maurin}); for varieties of codimension  $2$ in a torus (E. Bombieri, D. Masser and U. Zannier  \cite{BMZ1}) and in a product of elliptic curves with CM  (S. Checcoli, F. Veneziano and E. Viada \cite{TAI}).  Under  stronger geometric hypotheses on $V$, related results have been proved by  many other authors.

A  seminal result in this context has been the proof of the CIT 
 for  transverse curves in a torus,   by
 E. Bombieri, D. Masser and U. Zannier \cite{BMZ_tori}.  Their proof builds on two main statements.\\
\centerline{(a) The set of $C$-torsion anomalous points has bounded height.}

\centerline{(b) The set of $C$-torsion anomalous points has bounded degree.}
Then, by Northcott Theorem the set is finite. A novel idea in the   proof of part (b) is the use of a Lehmer type bound. An approach of similar nature has been used later in several papers. Unfortunately  the  Lehmer type bound is only known in CM abelian varieties, so the method does not work in non CM  abelian varieties. The author introduced a new method  which avoids the use of such a bound and  uses instead a  Bogomolov type bound which is known in some non CM abelian varieties.
In  \cite{ioant}, the author proved the TAC  for   curves in any power $E^N$ of an elliptic curve. Then, in \cite{ioimrn}  and \cite{ijnt} she  generalised the method. A method of a different nature, based on the Pila-Wilkie approach for counting points in  o-minimal structures, was introduced by J. Pila and U. Zannier \cite{PZ} and later exploited by several authors.\smallskip

For varieties of general dimension the situation is much more complicated. In this case the dimension of the torsion anomalous varieties can be bigger than zero, so for instance not all such anomalous varieties are maximal. We shortly explain the method used   for some varieties in tori  in \cite{BMZ1} and in products of CM elliptic curves in \cite{TAI}. In section \ref{BHC}, Theorem \ref{codimensionerelativauno} we give some more details and in section \ref{rangoN} we extend the bounds for curves. 

Let $Y$ be a $V$-torsion anomalous variety. So $Y$ is a component of an intersection $V\cap B$ with $B$ a minimal torsion variety.

\begin{itemize}
\item[(a)] Use the Zhang inequality \cite{ZhangEquidistribution} and the Arithmetic B\'ezout Theorem due to P. Philippon \cite{patrice} (both recalled in the next section) to get an upper bound for the essential minimum of type $$\mu(Y)\le c_1\frac{(\deg B)^{e_1}}{[k(Y):k]\deg Y}$$ for some rational exponent $e_1$.
\item[(b)] If $Y$ has positive dimension, use the Bogomolov lower bound for $\mu(Y)$  and if  $Y$ is a point, use the Lehmer lower bound for $\mu(Y)=h(Y)$. These give a bound of type
$$c_2\left(\frac{\deg B}{[k_{\rm tor}(Y):k_{\rm tor}]\deg Y}\right)^{e_2}\le \mu(Y)$$ for some rational exponent $e_2$.

\end{itemize}
Clearly if $e_1< e_2$, then  these two bounds imply that $\deg B \le c_3$. In addition, there are only finitely many algebraic subgroups of $G$ of degree $\le c_3$. Thus, there are only finitely many $V$-torsion anomalous varieties. 
In particular,  for varieties of codimension $2$ in a torus and in  a product of CM elliptic curves $e_1<e_2$ holds.

 \section{Bounded Height Conjectures}
\label{BHC}

 It is well known that an effective proof of  the CIT  implies the   Effective Mordell-Lang Conjecture, which has a strong impact on mathematics and computer science. Obtaining effective versions of the Mordell-Weil Theorem would be an astounding break-through.

 The effectivity of the results is a key point in  Diophantine Geometry, being the difference between a purely theoretical statement and a theorem which could, in principle, provide a way to explicitly solve diophantine equations.  
  Let us underline the difference between the words `effective', `explicit' and  `quantitative' when talking about solving diophantine equations:  a result  is effective if it provides a method that, in principle, allows to find all solutions;  a result is explicit if it gives in a simple way such solutions.  In particular an effective result guarantees the existence of an algorithm to find all solutions, while an explicit result provides the algorithm itself. On the other hand, quantitative is a weaker concept and it means to bound only the number of solutions. Notwithstanding their importance, several central results in  Diophantine Geometry are, unfortunately, not effective.  Among them  are the  above mentioned Mordell-Lang Conjecture  and Mordell-Weil Theorem, whose proofs  do not provide an upper bound for the height of the points, and consequently we do not know how to find such points.  \smallskip

Similarly, the obstruction to an effective proof of  the TAC and of the CIT is due  to  the lack of effective bounds for the normalized height of the torsion anomalous varieties.  In the following we generalize  the  Bounded Height Conjecture (BHC), formulated by E. Bombieri, D. Masser and U. Zannier  \cite{BMZ_tori} and proven by P.  Habbeger (see \cite{hab} and \cite{philipp}). Their BHC says nothing   for varieties  that  are not transverse or that are transverse, but  covered by anomalous varieties. Indeed in these cases $V^{oa}$ is empty. In addition their statements are not effective.
\begin{con}[BHC]\label{BHCB} For an irreducible  variety  $V\subset G$ , the set $V^{oa} \cap \mathcal{B}_{\dim V}$ has bounded height.\end{con}
Recall that $V^{oa}$ is the   set of points of $V$ deprived of the $V$-anomalous varieties and $\mathcal{B}_{\dim V}= \cup_{\codim B\ge \dim V}B$ is the union of all algebraic subgroups of $G$ of codimension at least $\dim V$.
We state here a natural variant of this conjecture which has particular significance for weak-transverse varieties and remains open in its generality.

 \begin{con}[BHC'] For an irreducible  variety  $V\subset G$,  the   maximal $V$-torsion anomalous points have  bounded  height.\end{con} 
 
As shown in \cite{ijnt}, the BHC' together with a Bogomolov type bound is sufficient to prove the  CIT. The method is effective, thus an effective version of the BHC', implies the effective CIT. For curves in an abelian variety, the  BHC' is proved by G. R\'emond \cite{RemondIII}. His method relies on a generalized Vojta inequality, so it cannot be made effective. For transverse varieties in an abelian variety with a positive density of ordinary primes, the Bogomolov type bound is proved by A. Galateau \cite{galateau}. R\'emond's and  Galateau's bounds, together with the just mentioned result of the author \cite{ijnt} give the TAC for curves  in abelian varieties with a positive density of ordinary primes. Using the BHC of P. Habegger mentioned above, the Bogomolov type bound of A. Galateau and the above implication of the author, we can deduce the CIT for transverse varieties $V$ not covered by $V$-anomalous varieties in an abelian variety with a positive density of ordinary primes.

To emphasize the effectivity of the methods presented below we state the Effective Bounded Height Conjeture.

 \begin{con}[EBHC]\label{BHC'} For an irreducible  variety  $V\subset G$,  the   maximal $V$-torsion anomalous points have  effectively bounded  height.\end{con}

Effective bounds were known only for transverse curves in a torus \cite{BMZ_tori} and in a product of elliptic curve \cite{ioannali} (to pass from power to product of elliptic curve is an immediate step  explained for instance in \cite{ioant}: in a product the matrices representing a subgroup have  several zero entries). In \cite[Theorem 1.4]{TAI}   we prove the EBHC  for $V$-torsion anomalous varieties of relative codimension one in a product of CM elliptic curves. This implies the effective TAC for such torsion anomalous varieties and in particular the effective TAC for $V$ of codimension $2$. As usual we indicate by $\ll$ an inequality up to a multiplicative constant, moreover we indicate as a subindex the variables on which the constant depends.

\begin{thm}\label{codimensionerelativauno}
Let $V$ be a weak-transverse variety in a  product $E^N$ of elliptic curves with CM  defined over a number field $k$; let $k_\mathrm{tor}$ be the field of definition of all torsion points of $E$. Then the maximal $V$-torsion anomalous varieties $Y$ of relative codimension one are finitely many. In addition  their degrees and normalized heights are bounded as follows. For any positive real $\eta$, there exist  constants depending only on $E^N$ and $\eta$ such that:
\begin{enumerate}
\item\label{mainprima} if $Y$ is not a translate then
\begin{align*} h(Y) &\ll_\eta   (h(V)+ \deg V)^{\frac{N-1}{N-\dim V-1}+\eta},\\
\deg Y &\ll_\eta\deg V (h(V)+ \deg V)^{\frac{\dim V}{N-\dim V-1}+\eta};
\end{align*}
\item\label{mainseconda} if $Y$ is a translate of positive dimension then
\begin{align*} h(Y)&\ll_\eta {(h(V)+\deg V)}^{\frac{N-2}{N-\dim V-1}+\eta}{[k_\mathrm{tor}(V):k_\mathrm{tor}]}^{\frac{\dim V-1}{N-\dim V-1}+\eta},\\
\deg Y&\ll_\eta (\deg V){(\cV)}^{\frac{\dim V-1}{N-\dim V-1}+\eta}.\end{align*}
\item\label{mainterza} if $Y$ is a point then {its N\'eron-Tate height and its degree are bounded as}
\begin{align*}\hat{h}(Y)  &\ll_\eta (h(V)+\deg V)^{\frac{N-1}{N-\dim V-1}+\eta}[k_{\mathrm{tor}}(V):k_{\mathrm{tor}}]^{\frac{\dim V}{N-\dim V-1}+\eta},\\
[\qe(Y):\mathbb{Q}]
&\ll_\eta{(\cV)}^{\frac{(\dim V+1)(N-1)}{(N-\dim V-1)^2}+\eta}.
\end{align*}
\end{enumerate}
\end{thm}

 The method to prove the theorem is  effective, and we give explicit dependence on $V$. This proves the EBHC for relative codimension one.
In \cite{Pacific} we generalise the method to abelian varieties with CM. 

We sketch here the proof which follows  the method described in section 2. This method will be   extended and detailed in section \ref{rangoN}. 

We shall prove that there are only finitely many $B$ which give a $V$-torsion anomalous varieties of relative codimension one. To this aim it is enough to bound the degree of $B$.
Let $Y$ be a maximal $V$-torsion anomalous variety of relative codimension one. Then $Y$ is a component of 
$ V \cap B$ with $B$ minimal for $Y$.

\begin{itemize}
\item[(a)] Use    the Arithmetic B\'ezout Theorem and the Zhang inequality to get an upper bound for the essential minimum of $Y$.

The  
   {\bf Arithmetic B\'ezout Theorem} states that for 
$Z_1, \dots , Z_g$  irreducible components of $V\cap W$ with $V$ and $W$ irreducible projective varieties embedded in $\mathbb{P}^m$, we have 
$$ \sum_{i=1}^g h(Z_i)\le \deg V h(W) +\deg W h(V) +c \deg V \deg W,$$
where $c$ is a constant depending explicitly on $m$.

To apply this in our setting, we construct a torsion variety $B'$, obtained by deleting some of the equations defining $B$, such that  $Y$ is a component of $V\cap B'$, $\deg B' \le (\deg B)^{\frac{\dim V-\dim Y}{\codim B}}$ and $h(B')=0$. Let $k$ be a field of definition  for $V$ and for any abelian subvariety of $E^N$ (such a $k$ depends only on the field of definition of $V$ and $E$). Then, also all the  conjugates of $Y$ over $k_{\rm tor}$ are  irreducible components of $V\cap B'$. Let $S$ be the number of such conjugates. The Arithmetic B\'ezout Theorem gives
\begin{equation}\label{patrice}
\begin{split}S h(Y)&\le  \deg V \,\,h(B') + \deg B' \left(h(V) + c \deg V\right)\\
& \le  (\deg B)^{\frac{\dim V-\dim Y}{\codim B}} \left(h(V) + c \deg V\right).\end{split}\end{equation} 

The {\bf Zhang inequality} states that   $${(\dim Y+1)^{-1}} \frac{ h(Y)}{\deg Y}\le \mu(Y)\le  \frac{ h(Y)}{\deg Y}.$$ So $\deg Y \mu(Y) \le h(Y)$ which, together with (\ref{patrice}),  gives  
$$S\deg Y \mu(Y)\le  \left(h(V) +c \deg V\right)\deg B^{\frac{\dim V-\dim Y}{\codim B}}.$$

\item[(b)] If $Y$ has positive dimension and it is not a translate, use the Bogomolov lower bound for $\mu(Y)$  and if  $Y$ is a point, use the Lehmer lower bound for $\mu(Y)=h(Y)$. In {\bf relative codimension one}, these bounds give respectively
$$c_2\frac{(\deg B)^{1-\eta}}{(\deg Y)^{1+\eta}}\le \mu(Y)$$ in positive dimension; and
$$c'_2\frac{(\deg B)^{1-\eta}}{[k_{\rm tor}(Y):k_{\rm tor}]^{1+\eta}}\le \mu(Y)$$ in dimension zero. The case of a translate can be reduced to the zero-dimensional case.
  \end{itemize}
  Parts (a) and (b) imply 
  $$\deg B ^{1-\eta}\le c_3 (h(V)+\deg V)\deg B^{\frac{\dim V-\dim Y}{\codim B}}.$$
  By the definition of $V$-torsion anomalous
\begin{equation}
\label{stella}N-\dim Y<N-\dim V+N-\dim B.
\end{equation} Thus $${\frac{\dim V-\dim Y}{\codim B}}<1$$ and  $$\deg B\le c(V),$$ with $c(V)$ a constant depending on $V$. Consequently there are only finitely many $V$-torsion anomalous varieties of relative codimension one and their height is effectively bounded using, for instance, the bound for $\deg B$ in  (\ref{patrice}).
 
Note that if   $\dim V=N-2$, then (\ref{stella}) implies $\dim B=\dim Y+1$. So for $V$ of codimension $2$ all $V$-torsion anomalous varieties have relative codimension one, and  the above proof covers all cases.

\medskip

In spite of the fact that the method above is  effective,  the use of a Lehmer type bound  is a deep obstacle  when trying to make the result explicit. Moreover such a bound  is not known in the non CM case.

In \cite{Trans} we use a different method which avoids the use of a Lehmer type bound, and so works for the CM and non CM case. Our method is completely effective and it proves new cases of the EBHC. We prove 
 that,   in a product of elliptic curves, the maximal $V$-torsion anomalous points of relative codimension one have effectively bounded height. An important advantage of this method is that we manage to make it explicit. In the non CM case,  we  compute all constants and we obtain the only known explicit bounds. Even if the result is only for points, the bounds are very strong, in the sense that the bounds depend uniformly on the degree and height of $V$, but are independent of the field of definition of $V$. This independence is crucial for some applications.  In \cite[Theorem 1.1]{Trans}  we prove the following theorem.

\begin{thm}\label{transmain}
 Let $V$ be an irreducible variety embedded in $E^N$.
Then the set of maximal  $V$-torsion anomalous points of relative codimension one has effectively bounded N\'eron-Tate height. If $E$ {is non} CM the bound is explicit, 
we have
\[
\hat h(P)\leq C_1(N)h(V)(\deg V)^{N-1}+C_2(E,N)(\deg V)^{N}+C_3(E,N),
\]
where 
\begin{align*}
C_1(N)&=(N!)^{N} N^{3N-2} \left( \frac{3^{N^2+N+1}2^{2N^2+3N-1}(N+1)^{N+1}}{(\omega_N \omega_{N-1})^2} \right)^{N-1}\\
C_2(E,N)&=C_1(N)\left(\frac{3^N \log 2}{2}+12 N \log 2+N\log 3+6N h_{\mathcal W}(E)\right)\\
C_3(E,N)&=\frac{7N^2}{6}\log 2+\frac{N^2}{2}h_{\mathcal W}(E),
\end{align*}
 $ \omega_r=\pi^{r/2}/\Gamma(r/2+1)$ is the volume of the euclidean unit ball in $\mathbb{R}^r$, $h(V)$ is the  normalised  height  of $V$ and $h_{\mathcal W}(E)$ is the height of the Weierstrass equation of $E$ (if $E$ is given by $y^2=x^3+Ax+B$ then $h_{\mathcal W}(E)$ is the Weil height of the point $(1:A^{\frac{1}{2}}:B^{\frac{1}{3}})$).
\end{thm}

The proof of  this theorem   relies on a more classical approximation process in the context of the geometry of numbers. Let $P$ be a point as in Theorem \ref{transmain}; in particular $P$ is a component of $V\cap B$ for some torsion  variety $B$. We replace $B$ with an auxiliary translate of the form $H+P$ in such a way that $P$ is still a component of $V\cap (H+P)$, and the degree and height  of $H+P$ can be controlled in terms  of $\hat h(P)$.
Using the properties of the height functions and the Arithmetic B\'ezout Theorem, we can in turn control the height of $P$ in terms of the height and degree of $H+P$ itself. Combining carefully these inequalities leads to the desired result.

This construction is not difficult, but when making the constants explicit  we must  do deep computations related to several height functions and to the degrees of varieties  embedded  in a projective space. In order to keep the constants as small as possible, we must  adapt and simplify several classical arguments.

\section {Applications to the   Mordell-Lang Conjecture}\label{sezML}

 The toric  version  of the Mordell-Lang Conjecture has been extensively studied, also in its effective form, by many authors. A completely effective version in the toric case can be found in the book of E. Bombieri and W. Gubler \cite[Theorem 5.4.5]{BiGi}. However this says nothing about the $k$-rational points of a variety, because the $k$-points are not a finitely generated subgroup in a torus. The Mordell -Weil Theorem shows that the $k$-points of an abelian variety are a finitely generated group. But unfortunately an effective Mordell-Lang Conjecture in abelian varieties is not known. The method by C. Chabauty and R. Coleman, surveyed, for example, by W. McCallum and B. Poonen in  \cite{MP}, gives a sharp quantitative Mordell-Lang statement for a curve in its Jacobian and $\Gamma$ of rank less than the genus; it involves a Selmer group calculation and a good understanding of the Jacobian of the curve. In some special cases, the bounds are sharp, this allows an inspection that leads to finding all the points in the intersection of the curve with $\Gamma$ (see W. McCallum \cite{McCallum}).
The method by Y. Manin and A. Demjanenko, described in Chapter 5.2 of the book of J.P. Serre  \cite{serreMWThm}, gives an effective Mordell Theorem for curves with many independent morphisms to an abelian variety. As remarked by J.P. Serre, the method remains of difficult application.

In this section we discuss  several implications of our bounds on the height to the effective Mordell-Lang Conjecture in the elliptic case. In particular we would like to underline that with our method  the bounds are totally  uniform in $h(V)$ and $\deg(V)$, thus it is very simple to give examples of curves with effective, and even explicit bounds for the height of their  rational points. We will give some new explicit examples in the next section. This is a major advantage over  the just mentioned effective classical methods.

\medskip

Let $E$ be an elliptic curve defined over the algebraic numbers.
 Let $\Gamma$ be a subgroup of $E^N$, we denote by $\overline {\Gamma}$  the group  generated by all coordinates  of the points in $\Gamma$ as an ${\mathrm{End}}(E)$-module. 
 The height of a set is as usual the supremum of the heights of its points.

For simplicity, we write MLC for the Mordell-Lang Conjecture.

\subsection*{Consequences of Theorem \ref{codimensionerelativauno} on the effective MLC} 

 The following theorem is proven in \cite[Corollaries 2.1 and 2.2]{TAI} . It is a consequence of the main Theorem of \cite{TAI}, stated here as Theorem \ref{codimensionerelativauno}.
\begin{thm}
\label{ML1}
Let $E$ be a CM elliptic curve and $\Gamma$ be a subgroup of $E^N$, where $N\ge3$. Let $\Ci$ be a weak-transverse curve in  $E^N$. Assume that $\overline{\Gamma}$ has rank one as an ${\mathrm{End}}(E)$-module. Then, for any positive real $\eta$, there exists a constant $c_1$, depending only on $E^N$ and $\eta$, such that  the set $\Ci\cap \Gamma $  has N\'eron-Tate height bounded  as
$$\hat{h}(\Ci\cap \Gamma)  \le c_1 (h(\Ci)+\deg \Ci)^{\frac{N-1}{N-2}+\eta}[k_{\mathrm{tor}}(\Ci):k_{\mathrm{tor}}]^{\frac{1}{N-2}+\eta}.$$ 

Let $N=2$ and let $\Ci$ be a transverse curve in $E^2$. Assume $\overline{\Gamma}$ is generated by $g$ as an ${\mathrm{End}}(E)$-module. Then for any positive real $\eta$ there exists a constant $c_2$ depending only on $E^N$ and $\eta$, such that  the set $\Ci\cap \Gamma $  has N\'eron-Tate height bounded  as
$$\hat{h}(\Ci\cap \Gamma)  \le c_2[k_\mathrm{tor}(\Ci\times g):k_\mathrm{tor}]^{1+\eta}(h(\Ci)+(\hat{h}(g)+1)\deg \Ci)^{2+\eta}.$$

\end{thm}

\subsection*{Consequences of Theorem  \ref{transmain} on the effective MLC}
In \cite[Theorem 1.3]{Trans}  we  give the following application of Theorem \ref{transmain}.  We can immediately remark the nice improvement on the previous result: this bound does not depend neither on the generator $g$ of $\overline{\Gamma}$ nor on the field of definition of $\Ci$. In addition the dependence on $\eta$ disappears. In the non CM case the constants are explicit. The method can be made explicit also in the  CM case (see \cite{SFI}).

\begin{thm}
\label{transrangouno}
Let $E$ be an elliptic curve.
Let  $N\ge3$ and $\Ci$ be a weak-transverse curve in  $E^N$.  Assume $\Gamma$ is a subgroup of $E^N$ such that $\overline{\Gamma}$   has rank one as an ${\mathrm{End}}(E)$-module. Then,  the set $\Ci\cap \Gamma $  has N\'eron-Tate height effectively bounded.  If $E$ is non CM, we have that  
\begin{align*}
 \hat h(\Ci\cap \Gamma)\leq C_1(N)h(\Ci)(\deg \Ci)^{N-1}+ C_2(E,N) (\deg \Ci)^{N}+C_3(E,N),
 \end{align*}
 where $C_1(N), C_2(E,N), C_3(E,N)$ are the same as in Theorem \ref{transmain}.

Let $N=2$ and  let $\Ci$ be a transverse curve in $E^2$. Then  the set $C\cap \Gamma $  has N\'eron-Tate height  effectively bounded.  If $E$ is non CM, we have that
 \begin{align*}
 \hat h(\Ci \cap \Gamma)\leq D_1h(\Ci)(\deg \Ci)^{2}+ D_2(E) (\deg \Ci)^{3}+D_3(E),
 \end{align*}
where
\begin{align*}
  D_1&=\frac{2^{64}3^{40}}{\pi^{8}} &\approx& 2.364\cdot 10^{34}\\
    D_2(E)&=\frac{2^{62}3^{41}}{\pi^{8}}\left(71\log 2+4\log 3+30h_{\mathcal W}(E)\right) &\approx& \left(5.319 \cdot h_{\mathcal W}(E) +9.504\right)\cdot 10^{35}\\
  D_3(E)&=\frac{9}{2}h_{\mathcal W}(E)+\frac{21}{2}\log 2 &\approx& 4.5\cdot h_{\mathcal W}(E)+7.279.
\end{align*}
In particular, in both cases, if $k$ is a field of definition for $E$ and $E(k)$ has rank 1, then all points in $\mathcal{C}(k)$  have N\'eron-Tate height effectively bounded as above. 

\end{thm}

\section{An Explicit Bound for the height of the rational points on a new family of curves}

As an application of Theorem \ref{transrangouno},  we give  in \cite{Trans} the following  explicit bound for the height of the rational points on   a  specific family of curves of increasing genus.

Let $E$ be the elliptic curve defined by the equation $y^2=x^3+x-1$; the group $E(\Q)$ has rank 1 with generator $g=(1,1)$. We write 
\begin{align*}
y_1^2=x_1^3+x_1-1\\
y_2^2=x_2^3+x_2-1
\end{align*}
for the equations of $E^2$ in $\P_2^2$, using affine coordinates $(x_1,y_1)\times (x_2,y_2)$. Consider the family of curves $\{\Ci_n\}_n$ with $\Ci_n\subseteq E^2$ defined via the additional equation $$x_1^n=y_2.$$
In \cite[Theorem 1.4]{Trans}  we prove:
\begin{thm}\label{teoremacurveCn}
For every $n\geq 1$ and the just defined curves $\Ci_n$ we have 
\[\hat{h}(\Ci_n(\mathbb{Q)})\leq 8.253\cdot 10^{38} (n+1)^3.\]

\end{thm}

Our  explicit results cannot be obtained with the method used in \cite{TAI}. Such kind of  examples are particularly interesting because they give, at least in principle, a method to find all the rational points of the given curves. The bound is unfortunately still too big to be implemented. We are now working to improve the bounds and to find a way to implement the process of testing all the points of such height (see \cite{puntiraz}). This is going to  allow us to find all rational points on this family of curves. In particular, we expect that the only rational points on all the curves in the above family are the points $(1,1)\times (1,1)$ and $(1,-1)\times (1,1)$. 
 
We give here a new family of curves of increasing genus and  we explicitly bound the height of their rational points. This underlines that our method can be easily adapted to new families of curves.

The new example is the following.

Let $E$ be the elliptic curve defined by the Weierstrass equation
\[E: y^2=x^3-x-2.\]
With an easy computation one can check that 
$$h_\mathcal{W}(E)=\frac{\log 2}{3}.$$

In particular the curve {is non} CM because $j(E)\not\in \Z$.
Furthermore the group $E(\Q)$ has rank 1 with generator $g=(2,2)$ and no non-trivial torsion points; this can be checked on a database of elliptic curve data  (such as {http://www.lmfdb.org/EllipticCurve/Q}).

We write 
\begin{align*}
y_1^2=x_1^3-x_1-2\\
y_2^2=x_2^3-x_2-2
\end{align*}
for the equations of $E^2$ in $\P_2^2$, using affine coordinates $(x_1,y_1)\times (x_2,y_2)$. We consider the family of curves $\{\Ci_n\}_n$ with $\Ci_n\subseteq E^2$ defined via the additional equation $$x_1^n+1=y_2.$$ 

We have the following: 
\begin{thm}\label{teoremacurveCn}
For every $n\geq 1$ and the just defined curves $\Ci_n$ we have 
\[\hat{h}(\Ci_n(\mathbb{Q)})\leq 9.689 \cdot 10^{38} (n+1)^3.\]
\end{thm}

\medskip

\begin{proof} The proof follows the line of the proof of \cite[ Theorem 1.4]{Trans}.
Note that the only irreducible curves in $E^2$ which are not transverse are translates, so  curves of genus $1$. Thus, to show that our curves $\Ci_n$ are transverse we show that they are irreducible and of genus $>1$. We then bound the height and degree of the $\Ci_n$  and substitute them in Theorem \ref{transrangouno} getting the desired result.\\

{\bf The irreducibility.}
The irreducibility of the $\Ci_n$ is easily seen to be equivalent to the primality of the ideal generated by the polynomials $y_1^2-x_1^3+x_1+2$ and $(x_1^{n}+1)^2-x_2^3+x_2+2$ in the ring $\qe[x_1,x_2,y_1]$. This follows from an easy argument in commutative algebra. \\

{\bf The genus.}
We shall show  that each $\Ci_n$ has genus at least 2;  in fact we  prove  that $\Ci_n$ has genus $4n+2$.

Consider the morphism $\pi_n:\Ci_n\to\P_1$ given by the function $y_2$. The morphism $\pi_n$ has degree $6n$, because for a generic value of $y_2$ there are three possible values for $x_2$, $n$ values for $x_1$, and two values of $y_1$ for each $x_1$.

Let $\alpha_1,\alpha_2,\alpha_3$ be the three distinct roots of the polynomial $f(T)=T^3-T-2$; let also $\beta_1,\beta_2,\beta_3,\beta_4$ be the four distinct roots of the irreducible polynomial $27T^4+108T^2+104$, which are the values such that $f(T)-\beta_i^2$ has multiple roots.
The $\beta_i$ have degree $4$, and therefore they cannot be equal to $\alpha_j^n+1$, which have degree $3$.
Also for all $n$,   the three numbers $\alpha_j^n+1$ are distinct, otherwise the ratio $\alpha_i/\alpha_j$ would be a root of 1 with degree a divisor of $6$, and all cases are easily discarded.

The morphism $\pi_n$ {is ramified over $\beta_1,\beta_2,\beta_3,\beta_4,1,\alpha_1^n+1,\alpha_2^n+1,\alpha_3^n+1,\infty$.}
{Each of the points $\beta_i$ has} $2n$ preimages of index 2 and $2n$ unramified preimages. The point 1 has 6 preimages ramified of index $n$. The points $\alpha_i^n+1$ have 3 preimages ramified of index 2 and $6n-6$ unramified preimages.
The point at infinity is totally ramified.

By Hurwitz formula
\begin{align*}
2-2g(\Ci_n)&=\deg\pi_n (2-2g(\P_1))-\sum_{P\in \Ci_n}(e_P-1)\\
2-2g(\Ci_n)&=12n-(4\cdot 2n +6(n-1)+3\cdot 3 +6n -1)\\
g(\Ci_n)&=4n+2.
\end{align*}

Thus the family $\{\Ci_n\}_n$ is a family of transverse curves in $E^2$.\\

 {\bf The degree.}
We can compute the degree of $\Ci_n$ as an intersection product. {Let $\ell, m$ be the classes of lines of the two factors of $\P_2^2$ in the Chow group. Then the degree of $\Ci_n$ is bounded by multiplying the classes of the hypersurfaces cut by the equation $x_1^n+1=y_2$, which is $n\ell+m$, by the two Weierstrass equations of $E$, which are $3\ell$ and $3m$, and by the restriction of an hyperplane of $\P_8$, which is $\ell+m$.
In the Chow group
\[(n\ell+m)(3\ell)(3m)(\ell+m)=9(n+1)(\ell m)^2\] 
and then
\begin{equation}\label{A}\deg\Ci_n\le9(n+1).\end{equation}
}\\

 {\bf The normalized height.}

We estimate the  height of $\Ci_n$  using  Zhang's inequality $\mu(X) \le \frac{h(X)}{\deg X} \le (1+\dim X) \mu(X)$ and computing an upper bound for the essential minimum $\mu(\Ci_n)$ of $\Ci_n$.   To this aim, we construct an infinite set of points on $\Ci_n$ of bounded height. By the definition of essential minimum, this gives also an upper bound for $\mu(\Ci_n)$.

Let us recall some definitions. The 
essential minimum of $X$ is defined in the Zhang inequality as
\[
 \mu(X)=\inf\{\theta\in\mathbb{R}\mid\{P\in X\mid h_2(P)\leq\theta\}\text{ is Zariski dense in }X\},
\]
where the $h_2$ is defined as follows.
Let $\mathcal{M}_K$ be the set of  places of a number field $K$.
For a point $P=(P_0:\dotsb :P_m)\in \P_m(K)$ let \begin{equation}\label{defiH2}
h_2(P)=\sum_{v\text{ finite}}\frac{[K_v:\Q_v]}{[K:\Q]}\log \max_i \{\abs{P_i}_v\} +\sum_{v\text{ infinite}}\frac{[K_v:\Q_v]}{[K:\Q]}\log \left(\sum_i \abs{P_i}_v^2\right)^{1/2} 
\end{equation}
be a modified version of the height that differs from the Weil height 
\begin{equation}\label{defiH}
h(P)=\sum_{v\in\mathcal{M}_K}\frac{[K_v:\Q_v]}{[K:\Q]}\log \max_i \{\abs{P_i}_v\}
\end{equation}
at the archimedean places.
These heights are both well-defined,  they extend to $\overline\Q$ and it follows easily from their definitions that
\begin{equation}\label{stima_altez}
h(P)\leq h_2(P)\leq h(P)+\frac{1}{2}\log(m+1).\end{equation}

\medskip

 Let $Q_\zeta=((x_1,y_1), (\zeta,y_2))\in \Ci_n$, where $\zeta \in \overline{\mathbb{Q}}$ is a root of unity. Clearly there exist infinitely many such points on $\Ci_n$. Using the equations of $E$ and $\Ci_n$, we have:
\[h(\zeta)=0, h(y_2)\leq \frac{\log 6}{2}, h(x_1)\leq \frac{\log 24}{2n}, h(y_1)\leq \frac{\log 24}{n}+\frac{\log 6}{2}.\] 
Thus
 \[ h(x_1,y_1)\leq \frac{3\log 24}{2n}+\frac{\log 6}{2}, h(\zeta,y_2)\leq \frac{\log 6}{2}\]
and 
\[ h_2(x_1,y_1)\leq \frac{3\log 24}{2n}+\frac{\log 18}{2}, h_2(\zeta,y_2)\leq \frac{\log 18}{2}.\]
 So for all points $Q_\zeta$ we have
\[h_2(Q_\zeta)=h_2(x_1,y_1)+h_2(\zeta,y_2)\leq \frac{3\log 24}{2n}+\log 18.\]  By the definition of essential minimum, we deduce \[\mu(\Ci_n)\leq \log 18+\frac{3\log 24}{2n}\]
and by Zhang's inequality 
\begin{equation}\label{B} h(\Ci_n)\leq 2\deg\Ci_n\mu(\Ci_n)\le 18(n+1)\left(\log 18+\frac{3\log 24}{2n}\right).\end{equation}

We can therefore apply Theorem \ref{transrangouno} with $N=2$ and $h_\mathcal{W}(E)=\frac{\log 2}{3}$ to each $\Ci_n$, which gives, for 
 $P\in \Ci_n(\mathbb{Q})$
 \begin{equation*}\hat h(P)\leq\left(2.364\cdot 10^{34}h(\Ci_n)+1.074\cdot 10^{36}\deg \Ci_n\right)(\deg \Ci_n)^2.\end{equation*}

Substituting  (\ref{A})  and (\ref{B}) for the degree and  height  of the $\Ci_n$ we obtain the theorem. 
\end{proof}

\section{Further results on the Effective and Quantitative MLC} \label{rangoN} In this section we show how the method used in \cite{TAI} can be extended  to    obtain a  more general result  for a curve $C$ weak-transverse in a product of CM elliptic curves.  In \cite[Theorem 1.5]{Pacific}  with $l=1$, $A_1=E$, $g_1=1$  and $e_1=N$ we proved the bound below for the height. Here we give bounds also for  the degree of the field of definition of the torsion anomalous points and for the degree of their minimal torsion varieties. This yields a sharp bound  for the cardinality of our set and it  allows us to get not only   examples of the effective MLC with the rank of $\overline{\Gamma}$  larger than $1$, but also sharp bounds for the quantitative MLC. These applications are clarified in Theorems \ref{MLR}, \ref{MLtre} and \ref{teoremoneML} below. 

 In what follows, the notations are the same as in \cite{TAI}. For clarity, we recall the well known relation between  algebraic subgroups of $E^N$and matrices with coefficients in ${\rm End}(E)$.
\begin{remark} \label{matrici}

Let $B+\zeta$ be an irreducible torsion variety of $E^N$  of codimension $\codim\sotto=r$ and let $\pi_B:E^N\to E^N/B$ be the natural projection.
Then $E^N/B$ is isogenous to $E^r$; let $\varphi_\sotto:E^N \to E^{r}$ be the composition of $\pi_B$ and this isogeny.

We  associate  $\sotto$ with the morphism $\varphi_\sotto$. Then $\ker \varphi_\sotto=\sotto+\tau$ with $\tau$ a torsion  subgroup whose cardinality is absolutely bounded.
Obviously $\varphi_\sotto$ is identified with a matrix in $\mathrm{Mat}_{r\times N}(\mathrm{End}(E))$ of rank $r$, such that the degree of   $\sotto$  is essentially (up to constants depending only on $N$) the sum of  the squares of the  minors of $\varphi_\sotto$. By {Minkowski}'s theorem,  we can choose the matrix representing $ \varphi_\sotto$ so that the degree of   $\sotto$ is essentially the product of the  squares $d_i$ of the norms  of the  rows of the matrix .

In short  $B+\zeta$ is a component of the torsion variety given as the  zero set of forms $h_1,   \dots , h_r$,  which are the rows of $\varphi_\sotto$,  of degree    $d_i$. In addition  $$d_1  \cdots d_r \ll \deg (B+\zeta)\ll d_1  \cdots d_r  .$$ 
We  {assume} to have ordered the $h_i$  by increasing degree. 
\end{remark}


 We are ready to  prove the following theorem.

\begin{thm}
\label{curva2}
Let $ \Ci$ be a  weak-transverse curve  in $E^N$, where $E$ has CM. Let $k$ be a field of definition for $E$.  Then the set
$$\mathcal{S}( \Ci)=  \Ci \cap \left(\cup_{\codim H>\dim H} H\right)$$ is a finite set of effectively bounded N\'eron-Tate height and effectively bounded cardinality.
Here $H$ ranges over all algebraic subgroups of codimension larger than the dimension.

More precisely, the set $\mathcal{S}( \Ci)$ can be  decomposed as 
$$\mathcal{S}( \Ci)=  \Ci_{\rm Tor}\cup\bigcup_{r=\lfloor \frac{N}{2}+1\rfloor}^{N-1} \mathcal{S}_r( \Ci),$$ 
where $\Ci_{\rm Tor}$ are the torsion points on $\Ci$, $\lfloor \frac{N}{2}+1\rfloor$ is the integral part of $ \frac{N}{2}+1$,
 $$\mathcal{S}_r( \Ci)=   \Ci \cap \bigcup_{i=1}^{M_r} H_i$$  and 
this union is taken over $M_r$ algebraic subgroups $H_i$ with $\codim H_i=r$. Moreover, for  any real $\eta>0$,  there exist  constants depending only on $E^N$ and $\eta$ such that 
$$M_r\ll_\eta  {(\cC)}^{\frac{r(N-r)N}{2r-N}+\eta}$$
and
$$ \deg H_i \ll_\eta   {(\cC)}^{\frac{r(N-r)(N+2r-2)}{2(r-1)(2r-N)}+\eta} ([k( \Ci):k]\deg  \Ci)^{\frac{Nr}{2(r-1)}+\eta}.$$

Furthermore, if $Y_0\in\mathcal{S}_r( \Ci)$ we have
$$ \hat{h}(Y_0)  \ll_\eta (h( \Ci)+\deg  \Ci)^{\frac{r}{2r-N}+\eta}[k_{\mathrm{tor}}( \Ci):k_{\mathrm{tor}}]^{\frac{(N-r)}{2r-N}+\eta};$$
and $$[k(Y_0):\mathbb{Q}]\ll_\eta ([k( \Ci):k]\deg  \Ci)^{\frac{r}{r-1}+\eta}{(\cC)}^{\frac{r(N-r)}{(2r-N)(r-1)}+\eta};$$
 the cardinality $S_r$ of the points in $\mathcal{S}_r( \Ci)$ is bounded as
$$S_r\ll_\eta [k( \Ci):k]^{c_1}(\deg  \Ci)^{c_1+1+\eta}{(\cC)}^{c_2+\eta},$$
where 

\[ c_1=\frac{rN(2N+1)}{2(r-1)},\]
\[c_2=\frac{r(N-r)(2rN+2r-2+2N^2-N)}{2(2r-N)(r-1)}.\]

\end{thm}

\begin{proof}

 
We notice that the set $\Ci_{\rm Tor}$ of all torsion points in $ \Ci$ has height zero,  is finite and has cardinality effectively bounded by the quantitative Manin-Mumford Conjecture (see relation (\ref{torsione})). From now on, we will be concerned with points in $\mathcal{S}( \Ci)$ that are not torsion.

Clearly  all the points in the intersection $ \Ci\cap(\cup_{\codim H>\dim H}H) $ are   $ \Ci$-torsion anomalous. In addition since $ \Ci$ is  a weak-transverse curve each torsion anomalous point is maximal. 

Let $Y_0\in \mathcal{S}( \Ci)$ be a non-torsion point; then  $Y_0\in  \Ci\cap H$, with $H$ the minimal subgroup containing $Y_0$ with respect to the inclusion, and  $\dim H<\codim H$. This gives $N-\codim H<\codim H$, so $\frac{N}{2}<
\codim H$. Since $Y_0$ is non torsion then $\codim H<N$. Thus  $Y_0\in \mathcal{S}_r(\Ci)$ for some $\frac{N}{2} <r<N-1$ and $\codim H=r$. This shows that the decomposition of $\mathcal{S}(\Ci)$ is correct.

Let  $B+\zeta$ be a component of $H$ containing $Y_0$. Clearly $\dim B=\dim H$ and $Y_0\in  \Ci\cap (B+\zeta)$ with $B+\zeta$ minimal for $Y_0$ and the torsion point $\zeta$ in the orthogonal complement of $B$.

Notice that 
\begin{equation}\label{degH81}
\deg H\leq (\deg B) \mathrm{ord}(\zeta).
\end{equation}
In fact $B+\langle \zeta\rangle$ is an algebraic subgroup of dimension equal to $\dim H$, it contains $Y_0$ and it is contained in $H$; thus $B+\langle \zeta\rangle=H$, by the minimality of $H$.

The proof now follows the lines of the proof of Theorem \ref{codimensionerelativauno} given in \cite{TAI}. We proceed to bound $\deg B$ and, in turn, the height of $Y_0$, using the Lehmer Type Bound for CM abelian varieties  and the Arithmetic B\'ezout theorem. Then, using Siegel's lemma, we get a bound for $[k(Y_0):k]$ and for the order and the number of torsion points $\zeta$, providing also a bound for $\deg H$.

 Recall that $r=\codim B=\codim H$. 

 We first exclude the case $r=1$ and show that the case $N-r=1$ is covered by  Theorem \ref{codimensionerelativauno}. If $r=1$ then $\dim B<\codim B$ implies $\dim B=0$ and $N=\dim B+\codim B=1$ contradicting the weak-transversality of $ \Ci$.
 
The case $N-r=1$ corresponds to $\dim B=1$ and $Y_0$ of relative codimension one, that can be treated with  Theorem \ref{codimensionerelativauno}.

Thus we can assume $r>N-r\geq 2$.
Moreover $2r>N$, since by assumption $\codim B> \dim B$. 

By Remark \ref{matrici}, the variety  $\sotto+\torsione$ is a component of the zero set of  forms $h_1,\dots,h_{r}$ of increasing degrees $d_i$ and $$d_1 \cdots d_{r}\ll\deg \sotto=\deg (\sotto+\torsione)\ll d_1 \cdots d_{r}.$$
Consider the torsion variety defined as the zero set of $h_1$, and let $A_0$ be  one of its connected components containing  $B+\zeta$.

Then
\begin{equation}\label{degA0curva}
\deg A_0\ll d_1\ll(\deg B)^{\frac{1}{r}}.
\end{equation}

From the Lehmer Type Bound applied to $Y_0$ in $B+\zeta$, for every positive real $\eta$, we get
\begin{equation}\label{carrizosa81}
\hat{h}(Y_0)\gg_\eta \frac{(\deg B)^{\frac{1}{N-r}-\eta}}{[k_{\mathrm{tor}}(Y_0):k_{\mathrm{tor}}]^{\frac{1}{N-r}+\eta}}.
\end{equation}

Notice that all conjugates of $Y_0$ over $k_\mathrm{tor}( \Ci)$ are components of $ \Ci\cap A_0$. In addition all conjugates of $Y_0$ over $k_\mathrm{tor}(V)$ are in $V\cap (B+\zeta)$, so the number of components of $V\cap A_0$ of height $\hat{h}(Y_0)$ is at least $$[k_\mathrm{tor}(V,Y_0):k_\mathrm{tor}(V)]\geq \frac{[k_\mathrm{tor}(Y_0):k_\mathrm{tor}]}{[k_\mathrm{tor}(V):k_\mathrm{tor}]}.$$
Applying the Arithmetic B\'ezout theorem to  $ \Ci\cap A_0$  we have
\begin{equation}\label{ArBez81}
\frac{[k_{\mathrm{tor}}(Y_0):k_{\mathrm{tor}}]}{[k_{\mathrm{tor}}( \Ci):k_{\mathrm{tor}}]}\hat{h}(Y_0)\ll (h( \Ci)+\deg  \Ci) (\deg B)^{\frac{1}{r}}.
\end{equation}

From \eqref{carrizosa81} and \eqref{ArBez81} we get
$$
(\deg B)^{\frac{2r-N}{r(N-r)}-\eta}\ll_\eta \cC  [k_{\mathrm{tor}}(Y_0):k_{\mathrm{tor}}]^{\frac{1}{N-r}-1+\eta}.
$$
Since $2r>N$, $N-r>1$ and $ [k_{\mathrm{tor}}(Y_0):k_{\mathrm{tor}}]\geq 1$, for $\eta$ small enough we obtain
\begin{equation}\label{degB81}
 \deg B\ll_\eta {(\cC)}^{\frac{r(N-r)}{2r-N}+\eta}.
\end{equation}
So, from \eqref{ArBez81} we have
\begin{equation}
\label{hY081} [k_{\mathrm{tor}}(Y_0):k_{\mathrm{tor}}]\hat{h}(Y_0)\ll_\eta {(\cC)}^{\frac{r}{2r-N}+\eta};
\end{equation}
while, using the right-hand side of \eqref{ArBez81} as a bound for $\hat{h}(Y_0)$ and \eqref{degB81} we get $$
\hat{h}(Y_0)\ll_\eta (h( \Ci)+\deg  \Ci)^{\frac{r}{2r-N}+\eta}[k_{\mathrm{tor}}( \Ci):k_{\mathrm{tor}}]^{\frac{(N-r)}{2r-N}+\eta}$$
as required.

Having bounded $\deg B$ and $\hat{h}(Y_0)$, we now proceed to bound $[k(Y_0):k]$ only in terms of $ \Ci$ and $E^N$.

We use Siegel's Lemma to construct an algebraic subgroup $G$ of codimension  $1=\dim  \Ci$ defined over $k$, containing $Y_0$ and of controlled degree.  The construction is exactly as in \cite[Propositions 3 and 4]{ioannali}. We present the steps of the proof. 
 
 We know that $\mathrm{End}(E)$ is an order  in an imaginary quadratic field $L$ with ring of integers $\mathcal{O}$. 
By minimality of $B+\zeta$, the coordinates of  $Y_0=(x_1,\ldots,x_N)$ generate an $\mathcal{O}$-module $\Gamma$ of rank equal to the dimension of $B+\zeta$ which is $N-r$. Let $g_1, \dots, g_{N-r}$ be generators of the free part of $\Gamma$  which give the successive minima, and are chosen as in \cite[Proposition 2]{ioannali}. Then  $\hat{h}(\sum_j \alpha_j g_j)\gg \sum_i |N_L(\alpha_j)|\hat{h}(g_j)$ for coefficients $\alpha_j$ in $\mathcal{O}$.  In addition, like in the case of relative codimension one,  the torsion part is generated by a torsion point $T$ of exact order $R$.

Therefore  we can write
$$x_i=\sum_j\alpha_{i,j} g_j +\beta_i T$$
 with coefficients $\alpha_i,\beta_i\in\mathcal{O}$ and $$N_L(\beta_i)\ll R^2.$$
 
 As in   \cite[Proposition 2]{ioannali}, we have
\begin{equation}\label{aicurva}
\hat{h}(x_i)\gg \sum_j \left|N_L(\alpha_{i,j})\right|\hat{h}(g_j).
\end{equation}

We define
$$
\nu_j=(\alpha_{1,j}, \dots, \alpha_{N,j}) \quad{\mathrm{and}} \quad  |\nu_j|=\max_i|N_L(\alpha_{i,j})|.$$
Then 
\begin{equation}\label{nucurva}
|\nu_j|\ll \frac{\hat{h}(Y_0)}{\hat{h}(g_j)}.\end{equation}

\medskip

We want to find coefficients $a_i\in\mathcal{O}$ such that $\sum_i^N a_i x_i=0$.
This gives a linear system of $N-r+1$ equations, obtained equating to zero the coefficients of $g_j$ and of $T$. The system has coefficients in $\mathcal{O}$ and $N+1$ unknowns: the $a_i$'s and one more unknown for the congruence relation arising from the torsion point.

We use the Siegel's Lemma over $\mathcal{O}$ as stated in \cite[Section 2.9]{BiGi}.  We get one equation with coefficients in $\mathcal{O}$;  multiplying it by a constant depending only on $E$ we may assume that it has coefficients in $\mathrm{End}(E)$. Thus it defines the sought-for algebraic subgroup $G$ of degree
$$\deg G\ll \left((\max_i N_L(\beta_i))\left(\prod^{N-r}_j|\nu_j|\right) \right)^{\frac{1}{r}}.$$
Let $G_0$ be  a $k$-irreducible component of $G$ passing through $Y_0$.
Then $$\deg G_0\ll \left(R^2\prod^{N-r}_j|\nu_j| \right)^\frac{1}{r}.$$

Since $ \Ci$ is weak-transverse, the point  $Y_0$ is a component of $ \Ci\cap G_0$.   In addition $ \Ci$ and $G_0$ are defined over $k$ and  B\'ezout's theorem gives  
\[[k(Y_0):k]\leq [k( \Ci):k]\deg  \Ci\deg G_0.\]
Hence
\[
[k(Y_0):k]\ll [k( \Ci):k]\deg  \Ci \left(R^2\prod^{N-r}_j|\nu_j| \right)^\frac{1}{r}.
\]
Using \eqref{nucurva} we get
\begin{equation}\label{degY081}
[k(Y_0):k]\ll 
[k( \Ci):k]\deg  \Ci \left(R^2\frac{\hat{h}(Y_0)^{N-r}}{\prod^{N-r}_j\hat{h}({g_j})}\right)^\frac{1}{r}.
\end{equation}

Following step by step the proof of Proposition 4 in \cite{ioannali}, using the Lehmer Type Bound for CM abelian variety in place of \cite[Theorem 1.3]{davidhindry} , we get
\begin{equation}\label{prod81}
\prod_{i=1}^{N-r}\hat{h}(g_i)\gg_\eta \frac{1}{[k_{\mathrm{tor}}(Y_0):k_{\mathrm{tor}}]^{1+\eta}}.
\end{equation}

By a result of Serre, recalled also in \cite[Corollary 3]{ioannali},  we know $[k(Y_0):k]\gg_\eta R^{2-\eta}$. Moreover from \eqref{hY081}, the product $\hat{h}(Y_0)[k_{\mathrm{tor}}(Y_0):k_{\mathrm{tor}}]$ is bounded. 
Substituting these bounds in \eqref{degY081} and recalling that $r>2$, for $\eta$ small enough we obtain
\begin{equation}\label{ultimoY0}
[k(Y_0):k]\ll_\eta ([k( \Ci):k]\deg  \Ci)^{\frac{r}{r-1}+\eta}{(\cC)}^{\frac{r(N-r)}{(2r-N)(r-1)}+\eta}.
\end{equation}

Moreover, as in the proof of \cite[Theorem 6.1]{TAI} , we can choose  $\zeta$ so that $$[k(\zeta):k]\ll [k(Y_0):\qe], $$ and 
\begin{equation} \label{ordzeta}\mathrm{ord}(\zeta)\ll_\eta [k(Y_0):\mathbb{Q}]^{\frac{N}{2}+\eta}. \end{equation}

Substituting \eqref{ordzeta} and \eqref{degB81} in \eqref{degH81}, we get the bound for $\deg H$.

We now prove the bound on $S_r$. 
By Remark \ref{matrici} and Minkowski's Theorem, we see that  the number of  abelian subvarieties  $G$ in $E^N$ of codimension $r$ and degree at most $\deg B$ is $\ll_\eta({\deg B})^{N}$.  In fact, if $G$ is such an abelian subvariety, by Minkowski's theorem there exists a linear trasformation of absolutely bounded degree that, up to reordering of the columns, transforms the matrix of $G$ in a matrix of the form 
\begin{equation*}
\phi= 
 \left(\begin{array}{cccccc}
d_1&\dots &0&*&\dots &*\\
\vdots & \ddots& \vdots & \vdots& *& \vdots\\
0&\dots &d_r&*&\dots &*
\end{array}\right),
\end{equation*} 
with  $|d_i|$  the maximum of the $i$-row and $\prod_i d_i\ll \deg B$. We now count the number of such matrices.
We have $\left(\prod_i 2|d_i|\right)^{N-r}\ll (\deg B)^{N-r}$ possibilities 
for $*$ and at most
$2\deg B$ possibilities for each $d_i$. Thus the number of such matrices is $\ll (\deg B)^{N}$. So   the number $D$ of  abelian subvarieties $G$ of degree at most $\deg B$ is  $\ll (\deg B)^{N}$.

As for the point $\zeta$, it is well known that the number of torsion points in $E^N$ of order bounded by a constant $T$ is at most $T^{2N+1}$. In fact the number of points of order dividing a positive integer $i$ is $i^{2N}$; so a bound for the number of torsion points of order at most $T$  is given by $$\sum_{i=1}^T i^{2N}\ll T^{2N+1}.$$

Applying B\'ezout's theorem to every intersection $V\cap (B+\zeta)$, we obtain that  the number $S_r$  is bounded by  $$S_r\ll{\deg  \Ci} \phantom{.}({\deg B})^{N+1} {\mathrm{ord}(\zeta)}^{2N+1}$$ and combining this with \eqref{ordzeta} and \eqref{degB81} we obtain the desired bound for $S_r$.

 Finally,  we notice that the algebraic  subgroups $H$  of codimension $r$ can be taken in a finite set $\{H_1,\dotsc, H_{M_r}\}$ where the $H_i=G_i+E^N[{\rm{ord}}(\zeta)]$ with $G_i$  abelian varieties of degree $\le \deg B$  and  codimension $r$. A bound for $M_r$ can be given effectively as done above for $D$, obtaining $M_r\ll ({\deg B})^N$.  

\end{proof}

\subsection{Consequences of Theorem \ref{curva2} on the effective MLC}
In this section we use Theorem \ref{curva2} to obtain effective height bounds for the intersection of $ \Ci$ with a group $\Gamma$, such that $\overline{\Gamma}$ has  rank larger than $1$. The following corollary extends Theorem \ref{ML1} and it is proven in \cite[Corollary 1.6]{Pacific}  with $l=1$, $A_1=E$, $e_1=N$, $g_1=1$ and $t_1=t<N/2$.   
\begin{thm}
\label{MLR}
Let $ \Ci$ be a weak-transverse curve in  $E^N$ with $E$ an elliptic curve with CM.
Let $k$ be a field of definition for $E$. Let $\Gamma$ be a subgroup of $E^N$ such that $\overline{\Gamma}$ has   rank $ t<  N/2$. Then, for any positive $\eta$, there exists a  constant $c_3$ depending only on $E^N$ and $\eta$, such that  the set $$ \Ci\cap \Gamma $$ has  N\'eron-Tate height bounded  as
$$\hat{h}( \Ci\cap \Gamma)  \le c_3(h( \Ci)+\deg  \Ci)^{\frac{N-t}{N-2t}+\eta}[k_\mathrm{tor}( \Ci):k_\mathrm{tor}]^{\frac{t}{N-2t}+\eta}.$$  
\end{thm}

We notice that Theorems \ref{ML1} and \ref{MLR} are proved for weak-transverse curves. If we assume the transversality of $ \Ci$ we can relax the hypothesis on the rank of $\overline\Gamma$.

\begin{thm}\label{MLtre}
Let $ \Ci$ be a transverse curve in  $E^N$ with $E$ a CM  elliptic curve  defined over $k$. Let $\Gamma$ be a subgroup of $E^N$ such that the free part of the group of its coordinates  is an ${\mathrm{End}}(E)$-module of   rank $t\le N-1$, generated by $g_1, \dots, g_t$. Then, for any positive $\eta$ there exists a constant $c_4$ depending only on $E^N$ and $\eta$, such that the set $$ \Ci\cap \Gamma $$ has  N\'eron-Tate height bounded  as
$$\hat{h}( \Ci\cap \Gamma)  \le c_4 [k_\mathrm{tor}( \Ci\times g):k_\mathrm{tor}]^{\frac{t}{N-t}+\eta}(h( \Ci)+ (\hat{h}(g)+1)\deg  \Ci)^{\frac{N}{N-t}+\eta}$$ 
where $g=(g_1, \dots, g_t)$.
\end{thm}
\begin{proof}
 Consider the curve $ \Ci'= \Ci\times g$ in $E^{N+t}$. Since $ \Ci$ is transverse then $ \Ci'$ is weak-transverse. If a point $(x_1, \dots, x_N)$ is in $\Gamma$, then there exist
 $0\neq a_i\in{\mathrm{End}}(E)$, an $N\times t$ matrix $B$ with coefficients in  ${\mathrm{End}}(E)$ and a torsion point $\zeta \in E^N$ such that 
 $$(a_1x_1,\ldots,a_Nx_N)^{t}=B (g_1,\ldots,g_t)^{t}+\zeta.$$ Thus the point $(x_1,\dots x_N,g_1, \dots, g_t) $ belongs to the intersection $ \Ci'\cap H$ with $H$ the torsion variety of codimension $N$  and dimension $t$ in $E^{N+t}$ defined by the equations 
 $$(a_1x_1,\ldots,a_Nx_N)^{t}=B (y_{N+1},\cdots, y_{N+t})^{t}+\zeta.$$

 Thus, $ \Ci\cap\Gamma$ is embedded in $ \Ci'\cap \cup_{\dim H=t}H$ and  $\hat{h}( \Ci\cap \Gamma)\leq \hat{h}( \Ci'\cap \cup_{\dim H=t}H)$ for $H$ ranging over all algebraic subgroups of dimension $t$. If $N> t$ then $\codim H>\dim H$.

The bound for the height is then given by Theorem \ref{curva2} applied to $ \Ci'\subseteq E^{N+t}$, where $\deg  \Ci=\deg  \Ci'$ and $h( \Ci')\leq 2( h( \Ci)+\hat{h}(g)\deg  \Ci)$
by Zhang's inequality. 
\end{proof}

\subsection{Consequences of Theorem \ref{curva2} on the quantitative MLC} \label{QML}
In  \cite[Theorem 1.2 ]{remond},  G. R\'emond gives  a bound on the cardinality of the intersection $ \Ci\cap \Gamma$ for a transverse curve in $E^ N$ and $\Gamma$ a  $\mathbb{Z}$-module of rank $r$. He obtains the following bound $$\sharp( \Ci\cap \Gamma)\leq {c(E^N, \mathcal{L})}^{r +1}(\deg  \Ci)^{(r +1) N^{20}},$$ where $c(E^N, \mathcal{L})$ is a positive effective constant depending on $E^N$ and on the choice of an invertible, symmetric and ample sheaf $\mathcal{L}$ on $E^N$.

The Manin-Mumford Conjecture is a special case of the Mordell-Lang Conjecture.  Explicit bounds on the number of torsion points in $ \Ci$ are given, for instance,  by E. Hrushovski in \cite{hru} and by 
 S. David and P. Philippon  in \cite[Proposition 1.12]{sipaIMRP}. There they show that the number of torsion points  on a non-torsion curve $ \Ci$ is at most \begin{equation}\label{torsione}\sharp( \Ci\cap\mathrm{Tor}_{E^N})\le (10^{2N+53}\deg  \Ci)^{33},\end{equation}
where $\mathrm{Tor}_{E^N}$ is the set of all torsion points of $E^N$.

As another consequence of Theorem \ref{curva2}  we also get   a sharp bound for the number of non torsion points in $ \Ci\cap \Gamma$ for $ \Ci$ weak-transverse in $E^N$, which, together with the just mentioned bounds for the torsion, improves in some cases  the bounds of G. R\'emond. Notice that below we use the rank $t$  of $\overline{\Gamma}$, the $\mathrm{End}(E)$-module of the coordinates of $\Gamma$. To compare  with R\'emond's result, we can use the trivial relation $r<2Nt$ and $t<Nr$.

\begin{thm}\label{teoremoneML}
Let $ \Ci$ be a curve in $E^N$, where $E$ has CM  and is defined over a number field $k$. Let $\Gamma$ be a subgroup of $E^N$ such that the group $\overline\Gamma$ has rank $t$ as an ${\rm End}(E)$-module. Let $\sharp(\Ci\cap \Gamma_{\setminus {\rm Tor}})$ be the number of non-torsion points in the intersection $ \Ci\cap \Gamma$. Then, for every positive real $\eta$ there exist  constants $d_1, d_2,d_3, d_4$ depending only on $E^N$ and $\eta$, such that:
\begin{enumerate}
\item[(i)]\label{ml1} if $ \Ci$ is weak-transverse in $E^N$, $N>2$ and $t=1$, we have \begin{align*}\sharp(\Ci\cap \Gamma_{\setminus {\rm Tor}})\leq & d_1  {(\cC)}^{\frac{(N-1)(4N^2-N-4)}{2(N-2)^2}+\eta}\\
&\cdot (\deg  \Ci)^{\frac{2N^3-N^2+N-4}{2(N-2)}+\eta} [k( \Ci):k]^{\frac{N(N-1)(2N+1)}{2(N-2)}+\eta};
\end{align*}
\item[(ii)]\label{ml2} if $ \Ci$ is transverse in $E^2$ and $t=1$, we have \begin{align*}\sharp(\Ci\cap \Gamma_{\setminus {\rm Tor}})\leq &d_2 {([k_\mathrm{tor}( \Ci\times g):k_\mathrm{tor}](h( \Ci)+(\hat{h}(g)+1)\deg  \Ci))}^{29+\eta}\\
&\cdot (\deg  \Ci)^{22+\eta}[k( \Ci\times g):k]^{21+\eta}
\end{align*}
where $g$ is a generator of $\overline\Gamma$;
\item[(iii)]\label{ml3} if $ \Ci$ is weak-transverse in $E^N$ and $t<N/2$, we have \begin{align*}\sharp(\Ci\cap \Gamma_{\setminus {\rm Tor}})\leq & d_3 {(\cC)}^{\frac{t(N-t)(4N^2-2Nt+N-2t-2)}{2(N-2t)(N-t-1)}+\eta}\\
&\cdot (\deg  \Ci)^{1+\frac{N(2N+1)(N-t)}{2(N-t-1)}+\eta}[k( \Ci):k]^{\frac{N(2N+1)(N-t)}{2(N-t-1)}+\eta};
\end{align*}
\item[(iv)]\label{ml4} if $ \Ci$ is transverse in $E^N$ and $t\leq N-1$, we have 
\begin{align*} \sharp(\Ci\cap \Gamma_{\setminus {\rm Tor}})\leq & d_4 
(\deg  \Ci)^{1+\frac{(N+t)N(2N+2t+1)}{2(N-1)}+\eta}[k( \Ci\times g):k]^{\frac{(N+t)N(2N+2t+1)}{2(N-1)}+\eta}\\
&\cdot  {([k_\mathrm{tor}( \Ci\times g):k_\mathrm{tor}](h( \Ci)+(\hat{h}(g)+1)\deg  \Ci))}^{\frac{Nt(4N^2+2t^2+6Nt+N-t-2)}{2(N-t)(N-1)}+\eta},
\end{align*}
where $\overline\Gamma$ is generated by $g_1,\ldots,g_t$ and $g=(g_1,\ldots,g_t)$.
\end{enumerate}
\end{thm} 
\begin{proof}
Part (i) and (ii)  are proved in \cite[Theorem 6.2]{TAI}. To prove Part (iii) we remark that  $0 < t <\frac{N}{2}$ thus the codimension $r$ of a subgroup of minimal dimension containing  a point in  $(\Ci\cap \Gamma_{\setminus {\rm Tor}})$ satisfies $\frac{N}{2} < r<N$. This shows that $(\Ci\cap \Gamma_{\setminus {\rm Tor}})\subset \bigcup_{r=\lfloor \frac{N}{2}+1\rfloor}^{N-1}\mathcal{S}_{r}( \Ci)$. Then we use  the bound in  Theorem \ref{curva2} for $S_{r}$ which is the cardinality of  $\mathcal{S}_r( \Ci)$. To prove Part (iv)  we embed $\Ci$ in the weak-transverse curve $\Ci\times g \subset E^{N+t}$. Like above, we remark that  $(\Ci\cap \Gamma_{\setminus {\rm Tor}})$ is embedded in $ \bigcup_{r=N}^{N+t-1}\mathcal{S}_{r}( \Ci\times g)$ and $N>\frac{N+t}{2}$, because $t<N$. Then we use  the bound in  Theorem \ref{curva2} for $S_{r}$ which is the cardinality of $ \mathcal{S}_r( \Ci\times g)$.
\end{proof}

\section*{Acknowledgments} I kindly thank the Referee for his nice report.  I  thank Sara Checcoli and Francesco Veneziano for helping me with several computations and  for  reading this work. I  thank Patrik Hubschmid for a careful revision of this paper. I kindly thank the organisers Andrea Bandini and Ilaria Del Corso and the Scientific Committee of the Third Italian Number Theory Meeting held in Pisa 2015, for inviting me to give a talk. I  thank the FNS for the financial support.

\def\cprime{$'$}
\providecommand{\bysame}{\leavevmode\hbox to3em{\hrulefill}\thinspace}
\providecommand{\MR}{\relax\ifhmode\unskip\space\fi MR }
\providecommand{\MRhref}[2]{%
  \href{http://www.ams.org/mathscinet-getitem?mr=#1}{#2}
}
\providecommand{\href}[2]{#2}

Evelina Viada:\\
Department of mathematics \\
ETH Zurich\\
R\"amistrasse 101\\
CH-8092 Zurich, Swithzerland.\\
email: evelina.viada@math.ethz.ch

\end{document}